\documentclass[twoside,11pt,a4paper]{article}

\usepackage{fullpage}
\usepackage{amssymb}
\usepackage{amsmath}
\usepackage{amsthm}
\usepackage{latexsym}
\usepackage{color}
\usepackage{url}
\usepackage[authoryear]{natbib}
\usepackage{enumitem}
\usepackage[colorlinks=true,citecolor=blue,urlcolor=black]{hyperref}

\newtheorem{definition}{Definition}[section]
\newtheorem{theorem}[definition]{Theorem}
\newtheorem{lemma}[definition]{Lemma}
\newtheorem{proposition}[definition]{Proposition}

\newcommand{\divergence}{\mathop{\nabla \cdot}}

\numberwithin{equation}{section}

\title{Stochastic Keller–Segel System with Porous Medium Diffusion and Nonlinear Chemotactic Sensitivity}
\author{
	Yiming Jiang$^{\mathrm{a}}$, Haohang Li$^{\mathrm{b},*}$, Yawei Wei$^{\mathrm{a}}$
	\\[0.4em]
	\small\textit{$^{\mathrm{a}}$School of Mathematical Sciences and LPMC, Nankai University, Tianjin, 300071, China}
	\\
	\small\textit{$^{\mathrm{b}}$School of Statistics and Data Science, Nankai University, Tianjin, 300071, China}
	\\[0.3em]
}

\date{\today}

\begin{document}

\maketitle

\begingroup
\renewcommand{\thefootnote}{\fnsymbol{footnote}}
\footnotetext[1]{Corresponding author. E-mail: \texttt{1120240077@mail.nankai.edu.cn}}	

\endgroup

\begingroup
\renewcommand{\thefootnote}{}
\footnotetext{\textit{Email addresses:} 
	\texttt{ymjiangnk@nankai.edu.cn} (Yiming Jiang),
	\texttt{1120240077@mail.nankai.edu.cn} (Haohang Li),
	\texttt{weiyawei@nankai.edu.cn} (Yawei Wei)}
\addtocounter{footnote}{-1}
\endgroup

\begin{abstract}
	In this paper, we investigate a stochastic Keller--Segel system with porous medium diffusion and nonlinear chemotactic sensitivity on a bounded one-dimensional domain. The model describes cell aggregation in complex environments, where the dispersal of cells is governed by density-dependent diffusion $\Delta u^{[m]}$, reflecting the combined effects of porous media and population crowding, and the perception of chemoattractants follows Stevens' power law, leading to the nonlinear chemotactic sensitivity $\nabla\cdot(u\nabla v^{[a]})$. In addition, random environmental fluctuations are incorporated through multiplicative noise $u\,dW(t)$, which represents stochastic perturbations in population dynamics. For $a\geq1$ and $m\geq2a+1$, we establish the global existence of martingale solutions, uniform a priori estimates, and preservation of non-negativity. The condition $m\geq2a+1$ reveals a balance between nonlinear chemotactic aggregation and porous-medium diffusion: stronger sensing response requires stronger diffusion to prevent excessive aggregation. The proof combines a decoupled auxiliary system, energy estimates, and a stochastic Schauder--Tychonoff fixed point argument to overcome the difficulties caused by degenerate diffusion, nonlinear drift, and stochastic perturbations.
\end{abstract}


\vspace{2mm}

\section{Introduction}

In this paper, we obtain the global existence of the following stochastic Keller--Segel system with porous medium diffusion and nonlinear chemotactic sensitivity:

\begin{equation}\label{system}
	\begin{cases}
		du = \big(\Delta (u^{[m]})  - \chi \nabla \cdot (u \nabla v^{[a]}) \big) dt + u \, dW(t), & x\in \mathcal{O}, \ t \in (0,T], \\
		-\Delta v = u - v, & x\in \mathcal{O},\\
		\frac{\partial u}{\partial n} = \frac{\partial v}{\partial n} = 0, & x\in \partial\mathcal{O}, \\
		u(0, x) = u_0(x),
	\end{cases}
\end{equation}
where  $u$ denotes the density of the cells,  $v$ is the concentration of the chemical substances, $\chi > 0$ is the chemotactic sensitivity and $m, a \geq 1$ are parameters, and $\mathcal{O}$ is a bounded domain in $ \mathbb{R}$ with a smooth boundary $\partial \mathcal{O}$. Here, we adopt the notation $u^{[m]} = |u|^{m-1} u$. The term $ \Delta u^{[m]}  $ models the diffusion of bacteria through a porous medium. According to Stevens' power law in psychophysics \citep{stevens1957psychophysical,steingrimsson2006empirical}, the perceived intensity (or subjective magnitude) $\Psi(I)$ of a sensation is related to the physical intensity $I$ of the stimulus by a power function:
$$\Psi(I) = k I^{a},$$
where $k>0$ is a scaling constant that depends on the units of measurement, and $a\geq 1$ is an exponent that varies by sensory modality. Although the chemoattractant concentration is modeled by the variable $v$, the perceived chemical signal by the cells may not be linear. We model the perceived chemical stimulus by a power-law transformation of the chemical concentration. This leads to the nonlinear chemotactic flux $\nabla\cdot(u\nabla(v^{[a]}))$, where the exponent $a$ reflects the sensory response characteristic of the organism.

The Keller--Segel model is a fundamental framework for describing the collective movement of organisms such as bacteria and amoebae under the gradient of chemical signal concentration. It was initially proposed by \cite{keller1970initiation}.  \cite{patlak1953random} explained the aggregation of Dictyostelium discoideum from the perspective of pattern formation. Generally, the classic version of this model only considers linear diffusion and linear chemotactic effects. However, in reality, there is a need for more realistic descriptions of the model, including volume-filling effects and nonlinear response effects. For instance, when organisms live in porous media, the conditions for linear diffusion are not met, and chemotactic sensitivity is not linear in the perception of signal.  \cite{sugiyama2006global} obtained results on global existence and decay for deterministic Keller–Segel systems with porous medium diffusion $\Delta u^m$ and nonlinear chemotactic sensitivity $\nabla \cdot (u^{q-1}\nabla v)$, showing that degenerate diffusion can prevent blow-up. For nonlinear chemotactic sensitivity, \cite{aguilera2023global} proved global stability for a hyperbolic chemotaxis model with logarithmic sensitivity $\nabla\cdot (u\nabla \log v)$ under time-dependent boundary conditions, demonstrating that solutions converge to equilibrium as $t\rightarrow \infty$ without smallness restrictions on initial data. 

In recent years, the study of stochastic Keller--Segel models incorporating random environmental effects and external forces with porous medium diffusion has attracted significant research attention (see, e.g., \cite{efendiev2011well,huang2020time}). As a macroscopic model derived from the limit of microscopic dynamics (see \cite{stevens2000derivation}), the Keller--Segel model is based on fundamental conservation laws and Fick's law of diffusion; this may lead to the neglect of important features in underlying microscopic kinetics, such as molecular fluctuations and deviations from the mean value. Random perturbations and environmental noise are inherent in natural systems, interacting nonlinearly with the system to fundamentally change its dynamics. To construct a more realistic model, it is natural to incorporate stochastic perturbations arising from environmental fluctuations. Barbu, Da Prato, and R\"ockner's recent monograph \cite{barbu2016stochastic} focuses on the stochastic porous medium equation, and several key results from this book are used in the present work. \cite{dareiotis2021porous} proved the solvability of the one-dimensional stochastic porous medium equation with space-time white noise, while \cite{dareiotis2019ergodicity} analyzed the long-term behavior of solutions to this equation.

The present work is motivated by the interplay between stochastic environmental fluctuations, porous-medium diffusion, and nonlinear biological sensing. In the model \eqref{system}, the exponent $a$ describes the nonlinear response of cells to the perceived chemical signal, while $m$ characterizes the strength of density-dependent diffusion. Our result shows that when $m\geq2a+1$, the diffusive effect is sufficiently strong to prevent excessive aggregation induced by nonlinear chemotactic perception. From a biological perspective, this indicates that in fluctuating porous environments, enhanced sensitivity to chemical signals requires stronger dispersal effects to maintain population stability. The analysis of this model requires a refined treatment of the nonlinear coupling between degenerate diffusion, nonlinear chemotactic drift, and stochastic perturbations. In particular, the nonlinear interaction between the power-law chemotactic response and the porous-medium structure prevents a direct application of standard variational SPDE method. To overcome this difficulty, we introduce a decoupled auxiliary system, establish suitable uniform estimates for the nonlinear drift term, and construct solutions through a stochastic Schauder--Tychonoff fixed point argument. Compared with the stochastic chemotaxis model in \cite{mukherjee2025martingale}, where the chemotactic sensitivity is different, and the stochastic Keller--Segel system with linear chemotactic sensitivity considered by \cite{wang2025global}, the present work incorporates a power-law perception effect and leads to the new threshold condition $m\geq 2a+1$.

The rest of the paper is organized as follows. Section 2 introduces the notation, functional setting, and preliminary estimates used throughout the paper. Section 3 is devoted to the proof of global existence. In Section 4, we establish the preservation of non-negativity of the cell density. Finally, Section 5 concludes the paper.

\section{Preliminaries}

In this section, we introduce the functional framework, stochastic setting, and analytical tools used throughout the paper. We also formulate the necessary assumptions on the noise.

\begin{definition}[Bessel Potential Space]\label{not:bessel}
	Let $1 \leq p < \infty$ and $s \in \mathbb{R}$. The Bessel potential space (or Sobolev space of fractional order) $H_p^s(\mathbb{R}^d)$ is defined by
	\[
	H_p^s(\mathbb{R}^d) := \left\{ f \in \mathcal{S}'(\mathbb{R}^d) : \|f\|_{H_p^s} := \left\| \left( (1 + |\xi|^2)^{\frac{s}{2}} \hat{f} \right)^\vee \right\|_{L^p} < \infty \right\},
	\]
	where $\hat{f}$ and $f^\vee$ denote the Fourier transform and its inverse, respectively. For the domain $\mathcal{O}$, the space $H_p^s(\mathcal{O})$ is defined as the restriction of $H_p^s(\mathbb{R}^d)$ to $\mathcal{O}$ in the sense of distributions, equipped with the norm
	\[
	\| f \|_{H_p^s(\mathcal{O})} := \inf \left\{ \| g \|_{H_p^s(\mathbb{R}^d)} : g \in H_p^s(\mathbb{R}^d),\; g|_{\mathcal{O}} = f \right\}.
	\]
\end{definition}
Let $\mathcal O\subset\mathbb R^d$ be a bounded domain with smooth boundary. We formulate the problem within the variational framework based on a Gelfand triple $(V, H, V^*)$. Specifically, we choose $V := L^{m+1}(\mathcal{O})$, $H := H_2^{-1}(\mathcal{O})$, and $U := L^2(\mathcal{O})$. Identifying $H$ with its dual $H^*$ via the Riesz isomorphism, and using the continuous embeddings $H_0^{1,2}(\mathcal{O}) \subset L^2(\mathcal{O}) \subset L^{\frac{m+1}{m}}(\mathcal{O})$, we obtain the required Gelfand triple:
\begin{equation}\label{setting:Gelfand_triple_intro}
	V = L^{m+1}(\mathcal{O}) \subset H_2^{-1}(\mathcal{O}) \cong \left(H_0^{1,2}(\mathcal{O})\right)^* \subset \left(L^{m+1}(\mathcal{O})\right)^* = V^*.
\end{equation}

We recall a fundamental property regarding the Laplacian operator in our functional setting.

\begin{lemma}[\cite{liu2015stochastic}, Lemma 4.1.12]\label{lem:laplace_iso}
	The map $-\Delta : H_0^{1,2}(\mathcal{O}) \to H_2^{-1}(\mathcal{O})$ is an isometric isomorphism. In particular,
	\begin{equation*}
		\langle -\Delta u, -\Delta v \rangle_{H_2^{-1}} = \langle u, v \rangle_{H_0^{1,2}} \quad \text{for all } u,v \in H_0^{1,2}(\mathcal{O}). 
	\end{equation*}
	Furthermore, $(-\Delta)^{-1} : H^{-1}_2(\mathcal{O}) \to H_0^{1,2}(\mathcal{O})$ serves as the Riesz map for $H$, such that for every $x \in H$,
	\begin{equation*}
		\langle x, \cdot \rangle_{H^{-1}_2} = \langle (-\Delta)^{-1} x, \cdot \rangle_{H_0^{1,2}}.
	\end{equation*}
\end{lemma}

Using Lemma \ref{lem:laplace_iso}, for $u \in H$ and $w \in V$, the duality pairing can be represented via the inverse Laplacian:
\begin{equation}\label{dual:V*toV}
	\begin{aligned}
		_{V^*} \langle u,w \rangle_V = \langle u,w\rangle_{H} 
		= \langle (-\Delta)^{-1}u, (-\Delta)^{-1}w\rangle_{H_0^{1,2}(\mathcal{O})}
		= \int_{\mathcal{O}} \big((-\Delta)^{-1}u(x)\big) w(x) \, dx. 
	\end{aligned}
\end{equation}

For the stochastic framework, let $\mathfrak{A} = (\Omega, \mathcal{F}, \{\mathcal{F}_t\}_{t \geq 0}, \mathbb{P})$ be a complete probability space equipped with a right-continuous filtration. Let $W(t)$ be a cylindrical Wiener process on $U = L^2(\mathcal{O})$. 

A bounded linear operator $T: U \to H$ is called a Hilbert--Schmidt operator if $\sum_{k=1}^\infty \|T \psi_k\|_H^2 < \infty$ for any orthonormal basis $\{\psi_k\}_{k\in\mathbb{N}}$ of $U$. The space of all such operators is denoted by $L_2(U,H)$, equipped with the norm $\|T\|_{L_2(U,H)}^2 := \sum_{k=1}^{\infty} \|T \psi_k\|_H^2$.

The cylindrical Wiener process $W(t)$ can be formally expanded as $W(t) = \sum_{k=1}^\infty \mu_k \psi_{k} \beta_k(t)$, where $\{\beta_k\}_{k \in \mathbb{N}}$ is a sequence of mutually independent standard one-dimensional Brownian motions, and $\{\mu_k\}_{k \in \mathbb{N}}$ are positive coefficients. We choose $\{\psi_{k}\}_{k \in \mathbb{N}}$ to be the eigenfunctions of $-\Delta$ under Dirichlet or Neumann boundary conditions, with corresponding eigenvalues $\{\lambda_k\}_{k \in \mathbb{N}}$. 

By elliptic regularity, $\{\psi_k\}_{k \in \mathbb{N}} \subset H^2(\mathcal{O})$. Since $d =1$, the Sobolev embedding yields an absolute constant $C>0$ such that 
\[\|\psi_{k}\|_{L^\infty(\mathcal{O})} \leq C \|\psi_{k}\|_{H^{2}(\mathcal{O})} \leq C\|\Delta\psi_k\|_{L^2(\mathcal{O})} \leq C \lambda_k.\]
Consequently, $\|u\psi_{k}\|_{H^{-1}_2(\mathcal{O})}^2 \leq C^2 \lambda_k^2 \|u\|_{H^{-1}_2(\mathcal{O})}^2$.
Throughout this paper, we impose the following summability condition on the noise coefficients:
\begin{equation}\label{assum:noise}
	\sum_{k=1}^{\infty} \mu_k^2 \lambda_k^2 := C_1 < \infty.
\end{equation}
Under assumption \eqref{assum:noise}, for $u \in L^{m+1}(\mathcal{O}) \subset H_2^{-1}(\mathcal{O})$, the Hilbert--Schmidt norm of $B(u) = u$ is well-defined and bounded:
\begin{equation}
	\|B(u)\|_{L_2(U,H)}^2 = \sum_{k=1}^{\infty} \|\mu_k u \psi_{k} \|_{H_2^{-1}(\mathcal{O})}^2 \leq C_1 \|u\|_{H_2^{-1}(\mathcal{O})}^2 \leq C\|u\|_{L^{m+1}(\mathcal{O})}^2.
\end{equation}
Thus, $B(u) \in L_2(U, H)$ and the stochastic integral is rigorously defined. We can now formulate the  concept of martingale solutions to the SPDE \eqref{system}.
\begin{definition}[Martingale Solution]
	We say that a tuple $(\mathfrak{A}, W(t), (u,v))$ is a martingale solution to the system \eqref{system} if 
	\begin{enumerate}[label=(\roman*)]
		\item $\mathfrak{A} = (\Omega, \mathcal{F}, \{\mathcal{F}_t\}_{t \geq 0}, \mathbb{P})$ is a probability space with a complete, right-continuous filtration;
		\item $W(t)$ is a cylindrical Wiener process on $L^2(\mathcal{O})$ with respect to $\mathfrak{A}$;
		\item $u: [0,T] \times \Omega \to H_2^{-1}(\mathcal{O})$ is an $\mathcal{F}_t$-progressively measurable process such that $(u,v)$ satisfies system \eqref{system} in the weak sense over $\mathfrak{A}$.
	\end{enumerate}
\end{definition}
To construct solutions to the truncated system in the variational framework, we will use the following standard existence and uniqueness result for stochastic evolution equations. The theorem provides a general criterion based on the Gelfand triple.
\begin{lemma}[\cite{liu2015stochastic}, Theorem 5.1.3] \label{lem:existence_uniqueness}
	
	Consider the stochastic equation
	\begin{equation} \label{eq:spde}
		{d}X(t) = A(t, X(t)) \, {d}t + B(t, X(t)) \, {d}W(t),
	\end{equation}
	where for some fixed time $T > 0$, the mappings
	\begin{equation*}
		A : [0, T] \times V \times \Omega \to V^*, \quad B : [0, T] \times V \times \Omega \to L_2(U, H)
	\end{equation*}
	are progressively measurable. Suppose that there exist constants $\alpha \in (1, \infty)$, $\beta \in [0, \infty)$, $\theta \in (0, \infty)$, $C_0 \in \mathbb{R}$ and a nonnegative adapted process $f \in L^{p/2}([0, T] \times \Omega; \mathrm{d}t \otimes P)$ for some $p \geq \beta + 2$. Assume that the following conditions hold for all $u, v, w \in V$ and $(t, \omega) \in [0, T] \times \Omega$:
	\begin{itemize}
		\item[(H1)] \emph{(Hemicontinuity)} The map $\lambda \mapsto {}_{V^*}\langle A(t, u + \lambda v), w \rangle_V$ is continuous on $\mathbb{R}$.
		
		\item[(H2)] \emph{(Local monotonicity)} 
		\begin{equation*}
			2 \, {}_{V^*}\langle A(t, u) - A(t, v), u - v \rangle_V + \|B(t, u) - B(t, v)\|_{L_2(U,H)}^2 \leq (f(t) + \rho(v))\|u - v\|_H^2,
		\end{equation*}
		where $\rho : V \to [0, \infty)$ is a measurable hemicontinuous function and locally bounded in $V$.
		
		\item[(H3)] \emph{(Coercivity)} 
		\begin{equation*}
			2 \, {}_{V^*}\langle A(t, v), v \rangle_V + \|B(t, v)\|_{L_2(U,H)}^2 \leq C_0\|v\|_H^2 - \theta\|v\|_V^\alpha + f(t).
		\end{equation*}
		
		\item[(H4)] \emph{(Growth)} 
		\begin{equation*}
			\|A(t, v)\|_{V^*}^{\frac{\alpha}{\alpha-1}} \leq (f(t) + C_0\|v\|_V^\alpha)(1 + \|v\|_H^\beta).
		\end{equation*}
	\end{itemize}
	Then for any $X_0 \in L^p(\Omega; H)$, equation \eqref{eq:spde} has a unique solution $X(t)$ such that $X(0) = X_0$, and it satisfies
	\begin{equation*}
		\mathbb{E} \left( \sup_{t \in [0, T]} \|X(t)\|_H^p \right) < \infty.
	\end{equation*}
\end{lemma}

\section{Global existence}

This section is devoted to the proof of the global existence result. We first construct a solution operator associated with a decoupled auxiliary system and prove that it satisfies the assumptions of the stochastic Schauder--Tychonoff fixed point theorem. The existence of a fixed point then yields a martingale solution to the original system. The verification of the required properties is divided into three steps: the operator maps a bounded set into itself, is continuous, and is compact.

\begin{theorem}\label{theorem1}
	Assume $a \geq 1$ and $m \geq 2a+1$. Let the initial data $u_0 \in H^{-1}_2(\mathcal{O})$ satisfy $\mathbb{E}\big[\|u_0\|_{L^{m+1}}^{m+1}\big] < \infty$. Then, for any finite time horizon $T>0$, there exists a global martingale solution $(\mathfrak{A}, W(t), (u,v))$ to the system \eqref{system}. Moreover, there exists a constant $C>0$, depending only on the initial data and parameters, such that
	\begin{equation}\label{est:thm_1}
		\mathbb{E}\left[ \sup_{t \in [0,T]} \|u(t)\|_{H_2^{-1}}^{2} + 4\int_{0}^{T} \| u(t)\|_{L^{m+1}}^{m+1} \, dt \right] \leq C,
	\end{equation}
	and
	\begin{equation}\label{est:thm_2}
		\mathbb{E}\left[ \sup_{t \in [0,T]} \|u(t)\|_{L^{m+1}}^{m+1} + \int_{0}^{T} m^2(m+1)\big\| |u(t)|^{m-1} \nabla u(t)\big\|_{L^2}^2 \, dt \right] \leq C.
	\end{equation}
\end{theorem}
To prove Theorem \ref{theorem1}, we apply a stochastic version of the Schauder--Tychonoff fixed point theorem (as adapted in \cite{mukherjee2025martingale,wang2025global}). The main challenge lies in the nonlinear chemotaxis term. To handle this, we define a space of progressively measurable processes:
\begin{align*}
	\mathbb{X}_{\mathfrak{A}}(R_1,R_2) := \Bigg\{ \xi : [0,T]\times \Omega &\to H^{-1}_2(\mathcal{O}) \ \Big| \ 
	\mathbb{E}\left[ \sup_{t \in [0,T]} \|\xi(t)\|_{H_2^{-1}}^{2} + 4\int_{0}^{T} \| \xi(t)\|_{L^{m+1}}^{m+1} \,dt \right] \leq R_1,\\
	&\mathbb{E}\left[ \sup_{t \in [0,T]} \|\xi(t)\|_{L^{m+1}}^{m+1} + \int_{0}^{T} C_m \| |\xi(t)|^{m-1} \nabla \xi(t)\|_{L^2}^2 \, dt \right] \leq R_2 \Bigg\},
\end{align*}
where $C_m = m^2(m+1)$. For a given $\xi \in \mathbb{X}_{\mathfrak{A}}(R_1,R_2)$, we consider the following decoupled linearised system:
\begin{equation}\label{auxiliary_system}
	\begin{cases}
		du = \big(\Delta (u^{[m]}) - \chi \nabla \cdot (\xi \nabla v^{[a]}) \big)dt + u \,dW(t), \\
		-\Delta v = \xi - v.
	\end{cases}
\end{equation}

By elliptic regularity, we can determine $v$ uniquely from $\xi$ and denote it by $v_\xi$. Using the Sobolev embedding $W^{2,2}(\mathcal{O}) \hookrightarrow W^{1,\infty}(\mathcal{O})$ for $d=1$, we have the a priori estimate:
\begin{equation}\label{ineq:v_infty}
	\|v_\xi\|_{L^\infty} + \|\nabla v_\xi\|_{L^\infty} \leq C\|\xi\|_{L^{2}}.
\end{equation}

\begin{lemma}\label{lemma1}
	Under the assumptions of Theorem \ref{theorem1}, for any given $\xi \in \mathbb{X}_\mathfrak{A}(R_1,R_2)$, there exists a unique solution $u$ to the auxiliary system \eqref{auxiliary_system} satisfying the same bounds as in \eqref{est:thm_1}.
\end{lemma}
Based on Lemma \ref{lemma1}, we can construct a solution operator $\mathcal{T}: \xi \mapsto u$. The existence of a martingale solution is then reduced to proving that $\mathcal T$ admits a fixed point.

\noindent\textbf{Proof of Lemma \ref{lemma1}.}
We show that there exists a solution $u$ to \eqref{auxiliary_system} by Lemma \ref{lem:existence_uniqueness}. Let us consider the Gelfand triple
$$V\subset H \cong H^* \subset V^*,$$
with $H:= H^{-1}_2 (\mathcal{O})$, then its dual space $H^* = H_0^{1,2}(\mathcal{O})$(corresponding to Neumann boundary conditions). Set $V:= L^{m+1}$ and its dual space $V^* = L^{\frac{m+1}{m}}$.

We let 
$$A(u):= \Delta u^{[m]} - \chi \divergence (\xi(t) \nabla v_\xi^{[a]}(t)), \quad B(u):= u,$$
where $v_\xi$ is the solution of the second equation of \eqref{auxiliary_system}, the equation can be rewritten as
\begin{align}
	\label{eq:sde}
	du(t) = A(u)dt + B(u) dW_t.
\end{align}

\noindent Verification of (H1) (Hemicontinuity): Let $u_1,u_2,w\in V,\omega \in \Omega$. We need to show that for $\lambda\in \mathbb{R}, |\lambda| \leq 1$,
$$\lambda \to _{V^*}\langle A(u_1+\lambda u_2),w\rangle_V$$
is continuous on $\mathbb{R}$. Since the map $\lambda \to (u_1 + \lambda u_2)^{[m]}$  is continuous, by the Dominated Convergence Theorem $\lambda \to \int_{\mathcal{O}}(u_1 + \lambda u_2)^{[m]}w dx$ is continuous.  Therefore, it follows from \eqref{dual:V*toV} that \mbox{$\lambda \to _{V^*}\langle A (u_1+\lambda u_2),w\rangle_V$} is continuous.

\noindent Verification of (H$2$) (Local monotonicity): Let $u,w \in V , t\in[0,T],\omega \in \Omega$. Since the function $s \mapsto |s|^{m-1}s $ is monotonically increasing. Then we get
\begin{align*}
	&\quad2\,{}_{V^*}\langle A(u) -A(w), u-w \rangle_V + \|u-w\|^2_{L_2(U,H)} \\
	&= -2 \int_{\mathcal{O}} (|u|^{m-1}u -|w|^{m-1} w)(u-w) dx + \|u-w\|^2_{L_2(L^2(\mathcal{O}) , H_2^{-1}(\mathcal{O}))}\\
	&\leq C\|u-w\|_{H_2^{-1}(\mathcal{O})}^2.
\end{align*}
Hence, $A(u)$ and $B(u)$ satisfy (H2) with $f(t)=C$, $\rho(w)=0$.

\noindent Verification of (H3) (Coercivity): Let $u\in V$, $t\in[0,T], \omega\in\Omega$. Using  Young's inequality with $\varepsilon$ and the Sobolev embedding $L^1{(\mathcal{O})}\hookrightarrow H^{-1}_{\frac{m+1}{m}}(\mathcal{O})$ for $d = 1,2$, we have
\begin{align*}
	&\quad2\,{}_{V^*}\langle A(u) , u \rangle_V + \|B(u)\|^2_{L_2(U,H)}\\
	&=-2\int_{\mathcal{O}} |u|^{m-1}u^2 dx -\chi \langle \divergence(\xi \nabla v^{[a]}_\xi), u \rangle_{H_2^{-1}(\mathcal{O})} + \|u\|^2_{L^2(L^2(\mathcal{O}),H_2^{-1}(\mathcal{O}))}\\
	&\leq -2\int_{\mathcal{O}} |u|^{m-1}u^2 dx + \chi \int_{\mathcal{O}} (-\Delta)^{-1}(\divergence(\xi \nabla v^{[a]}_\xi))  \cdot u dx + \|u\|^2_{L_2(L^2(\mathcal{O}),H_2^{-1}(\mathcal{O}))}\\
	&\leq -2\|u\|_{L^{m+1}(\mathcal{O})}^{m+1} +C_\chi \| (-\Delta)^{-1}(\divergence(\xi \nabla v^{[a]}_\xi)) \|_{L^{\frac{m+1}{m}}(\mathcal{O})}^{\frac{m+1}{m}} + \|u\|_{L^{m+1}(\mathcal{O})}^{m+1} +C \|u\|_{H_2^{-1}(\mathcal{O})}^2\\
	&\leq  -\|u\|_{L^{m+1}(\mathcal{O})}^{m+1} +C_\chi \|\xi\nabla v^{[a]}_\xi\|^{\frac{m+1}{m}}_{H^{-1}_{\frac{m+1}{m}}(\mathcal{O})} +C \|u\|_{H_2^{-1}(\mathcal{O})}^2\\
	&\leq -\|u\|_{L^{m+1}(\mathcal{O})}^{m+1} +C \|\xi \nabla v^{[a]}_\xi\|^{\frac{m+1}{m}}_{L^1(\mathcal{O})} +C \|u\|_{H_2^{-1}(\mathcal{O})}^2.
\end{align*}
We take $\theta=1,\alpha=m+1$ and $f(t)= C\|\xi(t) \nabla v_\xi^{[a]}(t)\|_{L^1(\mathcal{O})}^\frac{m+1}{m} $. It remains to check that $f\in L^1([0,T]\times\Omega,dt\otimes\mathbb{P})$
$$\mathbb{E}\left[\int_{0}^{T}\|\xi(t) \nabla v_\xi^{[a]}(t)\|^{\frac{m+1}{m}}_{L^1}dt\right]\leq C \mathbb{E}\left[\int_{0}^{T}\|\xi(t)\|^{(a+1)\frac{m+1}{m}}_{L^{m+1}}dt\right] \leq C \mathbb{E}\left[\int_{0}^{T}\|\xi(t)\|^{m+1}_{L^{m+1}}dt\right]^{\frac{a+1}{m}} \leq C R_1^\frac{a+1}{m}. $$

\noindent Verification of (H4) (Growth): Let $u \in V, t\in[0,T], \omega\in\Omega$. Since $\|A(u)\|_{V^*} = \sup\limits_{\|w\|_V=1} |{}_{V^*}\langle A(u) , w \rangle_V|$, we have
\begin{align*}
	|_{V^*}\langle A(u) , w \rangle_V| 
	&= \left|\int_{\mathcal{O}}u^{[m]} w \,dx -\chi \int_{\mathcal{O}} (-\Delta)^{-1} (\divergence (\xi \nabla v_\xi^{[a]})) dx\right|\\
	&\leq C\|u\|_{L^{m+1}}^m \|w\|_{L^{m+1}} +\chi \| (-\Delta)^{-1} (\divergence (\xi \nabla v_\xi^{[a]}))\|_{L^{\frac{m+1}{m}}} \|w\|_{L^{m+1}}\\
	&\leq C\|u\|_{L^{m+1}}^m \|w\|_{L^{m+1}} +\chi \| \xi \nabla v_\xi^{[a]}\|_{H^{-1}_{\frac{m+1}{m}}} \|w\|_{L^{m+1}}\\
	&\leq C\|u\|_{L^{m+1}}^m \|w\|_{L^{m+1}} + C \| \xi \nabla v_\xi^{[a]}\|_{L^{1}} \|w\|_{L^{m+1}}.
\end{align*}
Under the assumption $\xi\in \mathbb{X}_\mathfrak{A}(R_1,R_2)$, we take $\alpha=m+1, \beta=0$ and nonnegative adapted process 
$$f(t):= C \|\xi(t) \nabla v_\xi^{[a]}(t)\|_{L^1(\mathcal{O})}^{\frac{m+1}{m}} \in L^1([0,T]\times \Omega, dt\otimes \mathbb{P}) .$$
Then (H2), (H3), and (H4) of Lemma \ref{lem:existence_uniqueness} hold. Therefore, the model \eqref{auxiliary_system} has a unique solution $u(t)$ which satisfies
$$\mathbb{E} \left[\sup\limits_{0\leq s\leq T} \|u(s)\|_{H_2^{-1}(\mathcal{O})}^2 + \int_{0}^{T} \|u(s)\|_{L^{m+1}(\mathcal{O})}^{m+1} ds\right] < \infty.$$
\qed

\begin{proposition} \label{prop:1}
	For $\xi \in \mathbb{X}_\mathfrak{A}(R_1,R_2)$, let $u$ be the unique solution of the model \eqref{auxiliary_system}. For $a\geq 1$ and $m \geq 2a+1$, there exist numbers $R_1 >0$ and $R_2 >0$ such that $\mathcal{T} $ maps $\mathbb{X}_\mathfrak{A}(R_1,R_2)$ into itself.
\end{proposition}
\begin{proof}
The proof proceeds in two steps.

Step 1: Applying the It\^o formula \cite[Theorem 4.2.5]{liu2015stochastic} to the function $ \|u(t)\|^2_{H_2^{-1}}$, we get 
\begin{equation}
	\label{prop11}
	\begin{aligned}
		&\quad \|u(t)\|_{H_2^{-1}}^2 - \|u_0\|_{H_2^{-1}}^2 \\
		&= 2\int_0^t  {}_{V^*}\langle \Delta u^{[m]}(s), u(s) \rangle_V 
		- \chi  _{V^*}\left\langle  \divergence(\xi(s) \nabla v_\xi^{[a]}(s)), u(s) \right\rangle_V  ds \\
		&\quad +  \int_0^t \|u(s)\|^2_{L_2(L^2,H_2^{-1})}  ds 
		+ 2 \int_0^t  \langle u(s) , u(s)dW(s)\rangle_{H_2^{-1}}  \\
		&\leq - 2 \int_{0}^{t} \|u\|_{L^{m+1}}^{m+1} ds +2\chi \int_{0}^{t} |\int_{\mathcal{O}} (-\Delta)^{-1} (\xi(s,x)\nabla v_\xi^{[a]}(s,x)) u(s,x) dx| ds \\
		&\quad + C\int_{0}^{t}  \|u(s)\|_{H_2^{-1}}^2 ds + 2 \sum_{k=1}^{\infty} \int_{0}^{t} \langle u(s),\mu_k u(s) \psi_{k}\rangle_{H_2^{-1}} d\beta_k(s). 
	\end{aligned}
\end{equation}
Again using Young's inequality and the Sobolev embedding $L^1(\mathcal{O}) \hookrightarrow H_{\frac{m+1}{m}}^{-1} (\mathcal{O})$ to estimate the nonlinear term, we have
\begin{equation}
	\label{prop12}
	\begin{aligned}
		&\quad|\int_{\mathcal{O}} (-\Delta)^{-1} (\xi(s,x)\nabla v_\xi^{[a]}(s,x)) u(s,x) dx| \\
		&\leq  \|u(s)\|_{L^{m+1}} \|\xi(s,x)\nabla v_\xi^{[a]}(s,x)\|_{H_{\frac{m+1}{m}}^{-1}} \\ &\leq \|u(s)\|_{L^{m+1}}^{m+1} + C \|\nabla v_\xi^{[a]}(s,x)\|_{L^\infty}^{\frac{m+1}{m}} \|\xi(s)\|_{L^1}^{\frac{m+1}{m}} \\
		&\leq \|u(s)\|_{L^{m+1}}^{m+1} + C \|\xi\|_{L^{m+1}}^{(a+1)\frac{m+1}{m}}.
	\end{aligned}
\end{equation}
Take the supremum over \([0, t]\) in equation \eqref{prop11} and then take the expectation, we have 
\begin{equation}
	\label{prop13}
	\begin{aligned}
		&\quad \frac{1}{2}\mathbb{E} \sup\limits_{0\leq t\leq T} \|u(t)\|_{H_2^{-1}}^2 - \frac{1}{2}\mathbb{E}\|u_0\|_{H_2^{-1}}^2 \\
		&\leq -\mathbb{E} \int_{0}^{T} \|u(t)\|_{L^{m+1}}^{m+1} dt +C_\chi \mathbb{E} \int_{0}^{T}\|\xi(t)\|_{L^{m+1}}^{(a+1)\frac{m+1}{m}} dt \\
		&\quad+  C \mathbb{E}\int_{0}^{T}  \|u(t)\|_{H_2^{-1}}^2 dt + \mathbb{E} [\sup\limits_{0\leq t\leq T} \sum_{k=1}^\infty | \int_{0}^{t}\langle u(s),\mu_k u(s)\psi_{k}\rangle_{H_2^{-1}}d\beta_k(s) |]\\
		&\leq -\mathbb{E} \int_{0}^{T} \|u(t)\|_{L^{m+1}}^{m+1} dt +C_\chi T^{\frac{m-a-1}{m}}  \mathbb{E} [\int_{0}^{T}\|\xi(t)\|_{L^{m+1}}^{m+1} dt]^{\frac{a+1}{m}} \\
		&\quad+  C \mathbb{E}\int_{0}^{T}  \|u(t)\|_{H_2^{-1}}^2 dt + \mathbb{E} [\sup\limits_{0\leq t\leq T} \sum_{k=1}^\infty | \int_{0}^{t}\langle u(s),\mu_k u(s)\psi_{k}\rangle_{H_2^{-1}}d\beta_k(s) |].
	\end{aligned}
\end{equation}
The Burkholder-Davis-Gundy inequality with $p=1$ and Young's inequality show that 
\begin{equation}
	\label{prop14}
	\begin{aligned}
		&\quad\mathbb{E} [\sup\limits_{0\leq t\leq T} \sum_{k=1}^{\infty} |\int_{0}^{t}\langle u(s),u(s)\psi_{k}\rangle_{H_2^{-1}}d\beta_k(s) | ] \\
		&\leq \mathbb{E} \left[\left( \sum_{k=1}^{\infty} \int_{0}^{T}\langle u(t),\mu_k u(t)\psi_{k}\rangle^2_{H_2^{-1}}dt \right)^{\frac{1}{2}} \right] \\
		&\leq \mathbb{E} \left[\left( \sum_{k=1}^{\infty} \int_{0}^{T} \|u(t)\|_{H_2^{-1}}^2 \|\mu_k u(t) \psi_{k}\|_{H_2^{-1}}^2 dt \right)^{\frac{1}{2}}\right]\\
		&\leq \mathbb{E} \left[\left( \sup\limits_{0\leq t\leq T}\|u(t)\|_{H_2^{-1}}^2  \right)^{\frac{1}{2}} \left( C \int_{0}^{T}  \|u(t) \|_{H_2^{-1}}^2 dt \right)^{\frac{1}{2}}\right]\\
		&\leq \frac{1}{4} \mathbb{E} \sup\limits_{0\leq t\leq T} \|u(t)\|_{H_2^{-1}}^2 + C \mathbb{E} \left[ \int_{0}^{T}  \|u(t) \|_{H_2^{-1}}^2 dt \right].
	\end{aligned}
\end{equation}
Combining \eqref{prop13} and \eqref{prop14}, we obtain 
\begin{equation}
	\begin{aligned}
		&\quad\frac{1}{4}\mathbb{E} \sup\limits_{0\leq t\leq T} \|u(t)\|_{H_2^{-1}}^2 - \frac{1}{2}\mathbb{E}\|u_0\|_{H_2^{-1}}^2 \\
		&\leq -\mathbb{E} \int_{0}^{T} \|u(t)\|_{L^{m+1}}^{m+1} dt +C_\chi T^{\frac{m-a-1}{m}}  \mathbb{E} [\int_{0}^{T}\|\xi(t)\|_{L^{m+1}}^{m+1} dt]^{\frac{a+1}{m}} +  C \mathbb{E}\int_{0}^{T}  \|u(t)\|_{H_2^{-1}}^2 dt.
	\end{aligned}
\end{equation}
Rearranging gives
\begin{equation}
	\begin{aligned}
		&\quad\frac{1}{4}\mathbb{E} \sup\limits_{0\leq t\leq T} \|u(t)\|_{H_2^{-1}}^2 + \mathbb{E} \int_{0}^{T} \|u(t)\|_{L^{m+1}}^{m+1} dt\\  
		&\leq  \frac{1}{2}\mathbb{E}\|u_0\|_{H_2^{-1}}^2  +C_\chi T^{\frac{m-a-1}{m}} \mathbb{E} [\int_{0}^{T}\|\xi(t)\|_{L^{m+1}}^{m+1} ]^{\frac{a+1}{m}} dt +  C \mathbb{E}\int_{0}^{T}  \|u(t)\|_{H_2^{-1}}^2 dt.
	\end{aligned}
\end{equation}
Since \(\xi \in \mathbb{X}_\mathfrak{A}(R_1,R_2)\), by Gronwall's lemma, there exists a constant $C>0$ such that
\begin{equation}
	\begin{aligned}
		\mathbb{E} \sup\limits_{0\leq t\leq T} \|u(t)\|_{H_2^{-1}}^2 + 4\mathbb{E} \int_{0}^{T} \|u(t)\|_{L^{m+1}}^{m+1} dt \leq  e^{CT} [\mathbb{E}\|u_0\|_{H_2^{-1}}^2  +C_\chi T^{\frac{m-a-1}{m}} R_1^{\frac{a+1}{m}} ],
	\end{aligned}
\end{equation}
then we can choose $R_1$ large enough such that 
\begin{align}
	\mathbb{E} \sup\limits_{0\leq t\leq T} \|u(t)\|_{H_2^{-1}}^2 + 4\mathbb{E} \int_{0}^{T} \|u(t)\|_{L^{m+1}}^{m+1} dt \leq  R_1.	
\end{align}
Step 2: Applying It\^o's formula to $\|u(t)\|_{L^{m+1}}^{m+1}$ to estimate $\|u(t)\|_{L^{m+1}}^{m+1}$:
\begin{equation}
	\label{prop21}
	\begin{aligned}
		&\quad\|u(t)\|_{L^{m+1}}^{m+1} -\|u(0)\|_{L^{m+1}}^{m+1} \\
		&= (m+1)\int_{0}^{t}\left( \int_{\mathcal{O}}|u(s)|^{m-1}u(s) \Delta(u^{[m]}(s)) dx - \chi\int_{\mathcal{O}}  |u(s)|^{m-1}u(s) \divergence (\xi(s) \nabla v_\xi^{[a]}(s)) dx \right)ds\\
		&\quad + (m+1)\sum_{k=1}^{\infty} \int_{0}^{t} \int_{\mathcal{O}} |u(s)|^{m-1} u(s)  \cdot \mu_k u(s) \psi_{k} dx d\beta_k \\
		&\quad+ m(m+1)\int_{0}^{t}\int_{\mathcal{O}} |u(s)|^{m-1}  \sum_{k=1}^\infty |\mu_k u(s) \psi_{k}|^2 dx ds \\
		&\leq -m^2(m+1)\int_{0}^{t}\int_{\mathcal{O}} u^{2m-2}(s) |\nabla u(s)|^2  dx ds + \chi(m+1)\int_{0}^{t}\int_{\mathcal{O}} \nabla u^{[m]}(s) \xi(s) \nabla v_\xi^{[a]}(s) dx ds\\
		&\quad + (m+1)\sum_{k=1}^{\infty} \int_{0}^{t} \int_{\mathcal{O}} |u(s)|^{m+1}  \cdot \mu_k \psi_{k} dx d\beta_k + Cm(m+1) \int_{0}^{t} \|u(s)\|_{L^{m+1}}^{m+1} ds.
	\end{aligned}
\end{equation} 
Using the Young's inequality, we can write 
\begin{equation}
	\label{prop22}
	\begin{aligned}
		&\quad\chi(m+1) \int_{0}^{t}\int_{\mathcal{O}} \nabla u^{[m]}(s) \xi(s) \nabla v_\xi^{[a]}(s) dx ds \\
		&\leq \frac{m^2(m+1)}{2} \int_{0}^{t} \|\nabla u^{[m]}(s)\|_{L^2}^2 ds + C(m,\chi) \int_{0}^{t} \|\xi(s)\nabla v_\xi^{[a]}(s)\|_{L^2}^2 ds \\
		&\leq \frac{m^2(m+1)}{2}  \int_{0}^{t}\int_{\mathcal{O}} u^{2m-2}(s)|\nabla u(s)|^2 dx ds + C(m,\chi)\int_{0}^{t} \|\xi(s)\|_{L^2}^2 \|\nabla v_\xi^{[a]}(s)\|_{L^\infty}^2 ds.
	\end{aligned}
\end{equation}
By the Burkholder–Davis–Gundy inequality and Young's inequality, we have 
\begin{equation}
	\label{prop23}
	\begin{aligned}
		&\quad (m+1) \mathbb{E}\left[\sup\limits_{0\leq t\leq T}\sum_{k=1}^{\infty} \int_{0}^{t} \int_{\mathcal{O}} |u(s)|^{m+1}  \cdot \mu_k \psi_{k} dx d\beta_k  \right]  \\
		&\leq C(m+1) \mathbb{E} \left[\left(\int_{0}^{T}(\int_{\mathcal{O}} |u(t) |^{m+1} dx)^2 dt\right)^\frac{1}{2}\right] \\
		&\leq C(m+1) \mathbb{E} \left(\sup\limits_{0\leq t\leq T} \|u(t)\|_{L^{m+1}}^{m+1}\right)^\frac{1}{2} \left(\int_{0}^{T}\|u(t)\|_{L^{m+1}}^{m+1} dt \right)^\frac{1}{2} \\
		&\leq \frac{1}{2} \mathbb{E}  \sup\limits_{0\leq t\leq T} \|u(t)\|_{L^{m+1}}^{m+1} +C_m \mathbb{E} \int_{0}^{T} \|u(t)\|_{L^{m+1}}^{m+1} dt.
	\end{aligned}
\end{equation}

Combining  \eqref{prop22} and \eqref{prop23}, taking the supremum over \([0, T]\) in \eqref{prop21} and then the expectation, we obtain
\begin{equation}
	\label{prop24}
	\begin{aligned}
		&\quad\frac{1}{2}\mathbb{E} \sup\limits_{0\leq t\leq T} \|u(t)\|_{L^{m+1}}^{m+1} +\frac{m^2(m+1)}{2}\mathbb{E}\int_{0}^{T} \||u(t)|^{m-1} \nabla u(t)\|_{L^2}^2 dt \\
		&\leq \mathbb{E} \|u_0\|_{L^{m+1}}^{m+1} +C(\chi,m) \mathbb{E} \int_{0}^{T} \|\xi(t)\|_{L^2}^2 \|\nabla v_\xi^{[a]}(t)\|_{L^\infty}^2 dt + C_m \mathbb{E} \int_{0}^{T} \|u(t)\|_{L^{m+1}}^{m+1} dt \\
		&\leq \mathbb{E} \|u_0\|_{L^{m+1}}^{m+1} +C(\chi,m) T R_1^{\frac{2a+2}{m+1}} + C_m \mathbb{E} \int_{0}^{T} \|u(t)\|_{L^{m+1}}^{m+1} dt.
	\end{aligned}
\end{equation}
By Gronwall's inequality we get there exist $C >0$ such that
\begin{equation}
	\label{prop25}
	\begin{aligned}
		&\quad\mathbb{E} \sup\limits_{0\leq t\leq T} \|u(t)\|_{L^{m+1}}^{m+1} + {m^2(m+1)}\mathbb{E}\int_{0}^{T} \||u(t)|^{m-1} \nabla u(t)\|_{L^2}^2 dt \\
		&\leq  e^{CT} \left(2\mathbb{E} \|u_0\|_{L^{m+1}}^{m+1} +C(\chi,m) T R_1^{\frac{2a+2}{m+1}}\right).
	\end{aligned}
\end{equation}
By choosing $R_2$ sufficiently large, we have
\begin{align}
	\mathbb{E} \sup\limits_{0\leq t\leq T} \|u(t)\|_{L^{m+1}}^{m+1} + {m^2(m+1)}\mathbb{E}\int_{0}^{T} \||u(t)|^{m-1} \nabla u(t)\|_{L^2}^2 dt \leq R_2.
\end{align}
Summarizing, we have shown that there exists $R_1 >0$ and $R_2>0$ such that $\mathcal{T}$ maps $\mathbb{X}_\mathfrak{A}(R_1,R_2)$ into itself, which finishes the proof.
\end{proof} 
Next, we prove the continuity of the operator $\mathcal{T}$.

\begin{proposition}\label{prop:2}
	For all $\xi_1,\xi_2 \in \mathbb{X}_\mathfrak{A}(R_1,R_2)$, let $u_1 := \mathcal{T}(\xi_1), u_2:=\mathcal{T}(\xi_2)$ be the corresponding solutions of the model \eqref{auxiliary_system} with the same initial data $u_0$. For any $\varepsilon >0$, there exists $\delta >0$ such that if $\|\xi_1 -\xi_2\|_{\mathbb{X}_\mathfrak{A}} \leq \delta $, then 
	\begin{align*}
		\|\mathcal{T}(\xi_1) -\mathcal{T}(\xi_2)\|_{\mathbb{X}_\mathfrak{A}} \leq \varepsilon.
	\end{align*}
\end{proposition}
\begin{proof}
	We consider the difference \( u_1 - u_2 \), which satisfies
	\begin{align*}
		d(u_1 -u_2) = \left(\Delta (u_1^{[m]} - u_2^{[m]}) - \chi \divergence(\xi_1 \nabla v_{\xi_1}^{[a]} - \xi_2 \nabla v_{\xi_2}^{[a]})\right) dt + (u_1-u_2) dW(t).
	\end{align*}
	Applying It\^o's formula to $\|u_1(t) -u_2(t)\|_{H_2^{-1}(\mathcal{O})}^2$ and by canonical calculations we obtain
	\begin{align*}
		&\quad\frac{1}{2} \|u_1(t)-u_2(t)\|_{H_2^{-1}}^2 \\
		&\leq\int_{0}^{t}\langle u_1(s) -u_2(s), \Delta (u_1^{[m]}(s)-u_2^{[m]}(s)) -\chi \divergence(\xi_1(s) \nabla v_{\xi_1}^{[a]}(s) - \xi_2(s) \nabla v_{\xi_2}^{[a]}(s))  \rangle_{H_2^{-1}} ds \\
		&\quad+\sum_{k=1}^{\infty} \int_{0}^{t} \langle u_1(s)-u_2(s), (u_1(s) -u_2(s))\mu_k \psi_{k} \rangle_{H^{-1}_2} d\beta_k(s)\\
		&\quad+ \frac{1}{2}\int_{0}^{t } \|u_1(s)-u_2\|^2_{L_2(L^2,H_2^{-1})}ds \\
		&\leq -\int_{0}^{t} \|u_1(s)-u_2(s)\|_{L^{m+1}}^{m+1} ds  \\
		&\quad +\chi\int_{0}^{t}\int_{\mathcal{O}} (-\nabla)^{-1}(\xi_1(s) \nabla v_{\xi_1}^{[a]}(s) - \xi_2(s) \nabla v_{\xi_2}^{[a]}(s))(u_1(s)-u_2(s))dx ds \\
		&\quad + \sum_{k=1}^{\infty} \int_{0}^{t} \langle u_1(s)-u_2(s), (u_1(s) -u_2(s))\mu_k \psi_{k} \rangle_{H^{-1}_2} d\beta_k(s) \\
		&\quad + C \int_{0}^{t} \|u_1(s)-u_2(s)\|^2_{H_2^{-1}} ds .
	\end{align*}
	Taking the supremum over $[0,T]$ and taking expectations on both sides, we obtain
	\begin{align*}
		&\quad\frac{1}{2} \mathbb{E} \sup\limits_{0\leq t\leq T} \|u_1(t)-u_2(t)\|_{H_2^{-1}}^2 + \mathbb{E}\int_{0}^{T} \|u_1(t)-u_2(t)\|_{L^{m+1}}^{m+1} dt\\
		&\leq \chi \mathbb{E}\int_{0}^{T}\left| \int_{\mathcal{O}} (-\nabla)^{-1}(\xi_1(t) \nabla v_{\xi_1}^{[a]}(t) - \xi_2(t) \nabla v_{\xi_2}^{[a]}(t))(u_1(t)-u_2(t))dx \right|dt \\
		&\quad + \sum_{k=1}^{\infty}\mathbb{E} \sup\limits_{0\leq t\leq T} \left| \int_{0}^{t} \langle u_1(s)-u_2(s), (u_1(s) -u_2(s))\mu_k \psi_{k} \rangle_{H^{-1}_2} d\beta_k(s)\right| \\
		&\quad + C \int_{0}^{T} \|u_1(t)-u_2(t)\|^2_{H_2^{-1}} dt  \\
		&:= J_1 +J_2 +J_3.
	\end{align*}
	To estimate $J_1$, we carry out the following calculation:
	\begin{equation}
		\begin{aligned}
			&\quad \chi\mathbb{E}\int_{0}^{T}\int_{\mathcal{O}} (-\nabla)^{-1}(\xi_1(t) \nabla v_{\xi_1}^{[a]}(t) - \xi_2(t) \nabla v_{\xi_2}^{[a]}(t))(u_1(t)-u_2(t))dx dt \\
			&\leq \chi\mathbb{E} \left[\int_{0}^{T}\left| \int_{\mathcal{O}} (-\nabla)^{-1}((\xi_1(t)-\xi_2(t)) \nabla v_{\xi_1}^{[a]}(t) )(u_1(t)-u_2(t))dx \right|dt \right]\\
			&\quad+\chi\mathbb{E} \left[\int_{0}^{T} \left|\int_{\mathcal{O}} (-\nabla)^{-1}(\xi_2(t) \nabla (v_{\xi_1}^{[a]}(t) - v_{\xi_2}^{[a]})(t) )(u_1(t)-u_2(t))dx\right| dt \right]\\
			&:= I_1 +I_2.
		\end{aligned}
	\end{equation}
	By H\"older's inequality, the Young's inequality and the Sobolev embedding \[L^1(\mathcal{O}) \hookrightarrow H^{-1}_{\frac{m+1}{m}}(\mathcal{O}) ,\] there exists a constant $C>0$ with 
	\begin{align*}
		I_1(t) 
		&\leq \chi \mathbb{E}\left[ \int_0^T \left| \int_\mathcal{O} (-\nabla)^{-1}  \cdot \big( (\xi_1 - \xi_2) (\nabla v_{\xi_1}^{[a]}) \big) (u_1 - u_2) \, dx \right| dt \right] \\
		&\leq  \mathbb{E}\left[ \int_0^T C\left\|  (\xi_1(t) - \xi_2(t)) \nabla v_{\xi_1}^{[a]}(t) \right\|_{H^{-1}_\frac{m+1}{m}}^{\frac{m+1}{m}} + \frac{1}{2}\|u_1(t) - u_2(t)\|_{L^{m+1}}^{m+1}  dt \right] \\
		&\leq C \mathbb{E} \left[\int_{0}^{T} \|\xi_1(t)-\xi_2(t)\|_{L^2}^{\frac{m+1}{m}} \|\nabla v_{\xi_1}^{[a]}(t) \|_{L^\infty}^{\frac{m+1}{m}}  dt\right] + \frac{1}{2} \mathbb{E} \left[\int_{0}^{T}\|u_1(t) -u_2(t)\|_{L^{m+1}}^{m+1} dt\right] \\
		&\leq C \mathbb{E} \left[\int_{0}^{T} \|\xi_1(t)-\xi_2(t)\|_{L^2}^{\frac{m+1}{m}} \|\xi_1 \|_{L^{m+1}}^{a\frac{m+1}{m}}  dt\right] + \frac{1}{2} \mathbb{E} \left[\int_{0}^{T}\|u_1(t) -u_2(t)\|_{L^{m+1}}^{m+1} dt\right].
	\end{align*}
	Applying complex interpolation \cite[6.2.4 Theorem]{bergh2012interpolation}, we get for $\rho >0$ 
	\[
		(H^\rho_2(\mathcal{O}),H^{-1}_2(\mathcal{O}))_{\theta,2} = H^0_2(\mathcal{O}) = L^2(\mathcal{O})
	\]
	with $\rho(1-\theta) +\theta (-1) =0$ we know 
	\begin{align}
		\|\xi_1(t)-\xi_2(t)\|_{L^2} \leq \|\xi_1(t)-\xi_2(t)\|_{H^{-1}_2}^\theta \|\xi_1(t)-\xi_2(t)\|_{H^\rho_2}^{1-\theta}.
	\end{align} 
	Choosing $\rho \leq \frac{1}{m}$, it implies $\theta \leq \frac{1}{1+m}$. Using the H\"older inequality, for $p>1,q>1,r>1$ and $\frac{1}{p} +\frac{1}{q}+\frac{1}{r}=1$ we have 
	\begin{equation}
		\begin{aligned}
			I_1(t) 
			&\leq C \mathbb{E} \left[\int_{0}^{T} \left(\|\xi_1(t)-\xi_2(t)\|_{H^{-1}_2}^\theta \|\xi_1(t)-\xi_2(t)\|_{H^\rho_2}^{1-\theta}\right)^{\frac{m+1}{m}} \|\xi_1 \|_{L^{m+1}}^{a\frac{m+1}{m}}  dt\right] \\
			&\quad+ \frac{1}{2} \mathbb{E} \left[\int_{0}^{T}\|u_1(t) -u_2(t)\|_{L^{m+1}}^{m+1} dt\right]\\
			&\leq C \mathbb{E}\left[\left(\sup\limits_{0\leq t\leq T} \|\xi_1(t)-\xi_2(t)\|_{H^{-1}_2}^{\theta\frac{m+1}{m}}\right) \left(\int_{0}^{T}\|\xi_1(t)-\xi_2(t)\|_{H^\rho_2}^{(1-\theta)\frac{m+1}{m}} dt\right)\left(\sup\limits_{0\leq t\leq T}\|\xi_1\|_{L^{m+1}}^{a\frac{m+1}{m}}\right) \right]\\
			&\quad+ \frac{1}{2} \mathbb{E} \left[\int_{0}^{T}\|u_1(t) -u_2(t)\|_{L^{m+1}}^{m+1} dt\right]\\
			&\leq C(T) \mathbb{E} \left[\sup\limits_{0\leq t\leq T} \|\xi_1(t)-\xi_2(t)\|_{H^{-1}_2}^{p\theta\frac{m+1}{m}}\right]^{\frac{1}{p}} \mathbb{E}\left[ \int_{0}^{T}\|\xi_1(t)-\xi_2(t)\|_{H^\rho_2}^{q(1-\theta)\frac{m+1}{m}} dt \right]^{\frac{1}{q}}\\
			&\quad\times  \mathbb{E} \left[\sup\limits_{0\leq t\leq T}\|\xi_1\|_{L^{m+1}}^{r a\frac{m+1}{m}}\right]^\frac{1}{r}
			+ \frac{1}{2} \mathbb{E} \left[\int_{0}^{T}\|u_1(t) -u_2(t)\|_{L^{m+1}}^{m+1} dt\right].
		\end{aligned}
	\end{equation}
	Let $p=2m$, $q=2$ and $r=\frac{2m}{m-1}$, H\"older's inequality and Jensen's inequality yield
	\begin{equation}
		\begin{aligned}
			I_1 &\leq C(T) \mathbb{E} \left[\sup\limits_{0\leq t\leq T} \|\xi_1(t)-\xi_2(t)\|_{H^{-1}_2}^{2}\right]^{\theta\frac{m+1}{2m}} \mathbb{E}\left[ \int_{0}^{T}\|\xi_1(t)-\xi_2(t)\|_{H^\rho_2}^{2m} dt \right]^{(1-\theta)\frac{m+1}{2m^2}} \\
			&\quad \times \mathbb{E} \left[\sup\limits_{0\leq t\leq T}\|\xi_1\|_{L^{m+1}}^{ m+1}\right]^\frac{a}{m}
			+ \frac{1}{2} \mathbb{E} \left[\int_{0}^{T}\|u_1(t) -u_2(t)\|_{L^{m+1}}^{m+1} dt\right]\\
			&\leq C(T,m) \mathbb{E} \left[\sup\limits_{0\leq t\leq T} \|\xi_1(t)-\xi_2(t)\|_{H^{-1}_2}^{2}\right]^{\theta\frac{m+1}{2m}} \mathbb{E}\left[ \int_{0}^{T}\|\xi_1(t)\|_{H^\rho_2}^{2m} + \|\xi_2(t)\|_{H^\rho_2}^{2m} dt \right]^{(1-\theta)\frac{m+1}{2m^2}} \\
			&\quad \times\mathbb{E} \left[\sup\limits_{0\leq t\leq T}\|\xi_1\|_{L^{m+1}}^{ m+1}\right]^\frac{a}{m} + \frac{1}{2} \mathbb{E} \left[\int_{0}^{T}\|u_1(t) -u_2(t)\|_{L^{m+1}}^{m+1} dt\right].
		\end{aligned}
	\end{equation}
	
	\cite[Technical Proposition B.1]{mukherjee2025martingale}
	established that, for each \(i=1,2\),
	\begin{equation}
		\label{ineq:hausenblas}
		\int_0^T \|\xi_i(t)\|_{H_2^\rho}^{2m}\,dt
		\leq
		\int_0^T
		\bigl\|\xi_i^{m-1}(t)\nabla\xi_i(t)\bigr\|_{L^2}^{2}\,dt
		\leq R_2.
	\end{equation}
	Since $\xi_i \in \mathbb{X}_\mathfrak{A}(R_1,R_2),i=1,2$, we have
	\begin{align*}
		I_1 \leq C R_2^{(1-\theta)\frac{m+1}{2m^2}} R_1^\frac{a}{m} \mathbb{E} \left[\sup\limits_{0\leq t\leq T} \|\xi_1(t)-\xi_2(t)\|_{H^{-1}_2}^{2}\right]^{\theta\frac{m+1}{2m}} +\frac{1}{2} \mathbb{E} \left[\int_{0}^{T}\|u_1(t) -u_2(t)\|_{L^{m+1}}^{m+1} dt\right].
	\end{align*}
	Next, we estimate $I_2$. Again using Young's inequality and Sobolev embedding 
	\[L^1(\mathcal{O})\hookrightarrow H^{-1}_{\frac{m+1}{m}},\] we have
	\begin{equation}
		\begin{aligned}
			I_2 &\leq \chi\mathbb{E} \left[\int_{0}^{T} \left|\int_{\mathcal{O}} (-\nabla)^{-1}(\xi_2 \nabla (v_{\xi_1}^{[a]} - v_{\xi_2}^{[a]}) )(u_1(t)-u_2(t))dx\right| dt \right]\\
			&\leq C \mathbb{E} \left[\int_{0}^{T} \|\xi_2 \nabla (v_{\xi_1}^{[a]} - v_{\xi_2}^{[a]})\|_{L^1}^{\frac{m+1}{m}} + \frac{1}{4} \|(u_1(t)-u_2(t))\|_{L^{m+1}}^{m+1} dt \right].
		\end{aligned}
	\end{equation}
	Since the equation for $v$ is linear, the difference $v_{\xi_1} -v_{\xi_2}$ satisfies 
	\begin{align*}
		\begin{cases}
			-\Delta(v_{\xi_1} -v_{\xi_2}) = (\xi_1-\xi_2) - (v_{\xi_1} -v_{\xi_2}),\quad x\in\mathcal{O},\\
			\frac{\partial (v_{\xi_1}-v_{\xi_2})}{\partial n} =0, \quad x\in \partial \mathcal{O}.
		\end{cases}
	\end{align*}
	Then, using Sobolev embedding $W^{2,2}(\mathcal{O}) \hookrightarrow W^{1,\infty}(\mathcal{O})$ for $d=1$, we have
	\begin{align}
		\|v_{\xi_1}-v_{\xi_2}\|_{L^\infty} +\|\nabla(v_{\xi_1}-v_{\xi_2})\|_{L^\infty} \leq C \|\xi_1-\xi_2\|_{L^2}.
	\end{align}
	Moreover, by the chain rule,
	\[
	\nabla\big(|v_\xi|^{a-1}v_\xi\big)=a|v_\xi|^{a-1}\nabla v_\xi.
	\]
	Hence
	\begin{equation}
		\begin{aligned}
			&\quad\left\|\nabla\left(|v_{\xi_1}|^{a-1}v_{\xi_1}-|v_{\xi_2}|^{a-1}v_{\xi_2}\right)\right\|_{L^\infty} \\
			&\leq C(a)\Big(
			\|v_{\xi_1}\|_{L^\infty}^{a-1}\|\nabla(v_{\xi_1}-v_{\xi_2})\|_{L^\infty}
			+\||v_{\xi_1}|^{a-1}-|v_{\xi_2}|^{a-1}\|_{L^\infty}\|\nabla v_{\xi_2}\|_{L^\infty}
			\Big).
		\end{aligned}
	\end{equation}
	For $1<a <2$, we have 
	\begin{align}
		\label{ineq:a<2}
		\||v_{\xi_1}|^{a-1}-|v_{\xi_2}|^{a-1}\|_{L^2} \leq \| |v_{\xi_1}-v_{\xi_2}|^{a-1}\|_{L^2} \leq C(a) \|v_{\xi_1}-v_{\xi_2}\|_{L^2}^{a-1}.
	\end{align}
	For $ a\geq 2$, let $M(t)=max(\|{\xi_1}(t)\|_{L^{\infty}},\|{\xi_2}(t)\|_{L^{\infty}})$, we have
	\begin{align}
		\label{ineq:a>2}
		\||v_{\xi_1}|^{a-1}-|v_{\xi_2}|^{a-1}\|_{L^\infty} \leq C M(t)^{a-2}\|v_{\xi_1}-v_{\xi_2}\|_{L^\infty}.
	\end{align}
	It follows that for $1 < a <2$, we have 
	\begin{equation}
		\begin{aligned}
			I_2 &\leq  C(a) \mathbb{E} \Big[\int_{0}^{T} \|{\xi_1}(t)\|_{L^{m+1}}^{(a-1)\frac{m+1}{m}} \|{\xi_1}(t) -{\xi_2}(t)\|_{L^2}^{\frac{m+1}{m}}   \|\xi_2(t)\|_{L^{m+1}}^{\frac{m+1}{m}} dt\Big]\\
			& \quad +   C(a)\mathbb{E} \Big[\int_{0}^{T} \|\xi_1(t) -\xi_2(t) \|_{L^2}^{(a-1)\frac{m+1}{m}}  \| \xi_2(t)\|_{L^{m+1}}^{2\frac{m+1}{m}}  dt \Big] \\
			& \quad + \frac{1}{4}\mathbb{E} \Big[ \int_{0}^{T}\|(u_1(t)-u_2(t))\|_{L^{m+1}}^{m+1} dt \Big].
		\end{aligned}
	\end{equation}
	As in the estimation of \( I_1 \), by \eqref{ineq:hausenblas} and H\"older's inequality with $p=2m, q=2$ and  $r=\frac{2m}{m-1}$, we have
	\begin{align*}
		I_2 &\leq C(a,R_1,R_2,T) \mathbb{E} \left[\sup\limits_{0\leq t\leq T} \|\xi_1(t)-\xi_2(t)\|_{H_2^{-1}}^2\right]^{\theta\frac{m+1}{2m}} \\
		&\quad+ C(a,R_1,R_2,T) \mathbb{E} \left[\sup\limits_{0\leq t\leq T} \|\xi_1(t)-\xi_2(t)\|_{H_2^{-1}}^2\right]^{(a-1)\theta\frac{m+1}{2m}} \\
		&\quad+  \frac{1}{4}\mathbb{E} \Big[\int_{0}^{T} \|(u_1(t)-u_2(t))\|_{L^{m+1}}^{m+1} dt \Big].
	\end{align*}
	For $a\geq 2$, from \eqref{ineq:a>2} we obtain
	\begin{equation}
		\begin{aligned}
			I_2 &\leq C(a) \mathbb{E} \Big[\int_{0}^{T} \|{\xi_1}(t)\|_{L^{m+1}}^{(a-1)\frac{m+1}{m}} \|{\xi_1}(t) -{\xi_2}(t)\|_{L^2}^{\frac{m+1}{m}}   \|\xi_2(t)\|_{L^{m+1}}^{\frac{m+1}{m}} dt\Big]\\
			& \quad +   C(a)\mathbb{E} \Big[\int_{0}^{T} M(t)^{a-2} \|\xi_1(t) -\xi_2(t) \|_{L^2}^{\frac{m+1}{m}}  \| \xi_2(t)\|_{L^{m+1}}^{2\frac{m+1}{m}}  dt \Big] \\
			& \quad + \frac{1}{4}\mathbb{E} \Big[ \int_{0}^{T}\|(u_1(t)-u_2(t))\|_{L^{m+1}}^{m+1} dt \Big],
		\end{aligned}
	\end{equation}
	thus,
	\begin{align*}
		I_2 \leq C(a,R_1,R_2,T) \mathbb{E} \left[\sup\limits_{0\leq t\leq T} \|\xi_1(t)-\xi_2(t)\|_{H_2^{-1}}^2\right]^{\theta\frac{m+1}{2m}} + \frac{1}{4}\mathbb{E} \Big[\int_{0}^{T} \|(u_1(t)-u_2(t))\|_{L^{m+1}}^{m+1} dt \Big].
	\end{align*}
	Next, we estimate $J_2$. By the Burkholder-Davis-Gundy inequality and Young's inequality, we obtain 
	\begin{equation}
		\begin{aligned}
			J_2&\leq \sum_{k=1}^{\infty} \mathbb{E} \left[\left(\int_{0}^{T}\|u_1(t)-u_2(t)\|^2_{H_2^{-1}} \|(u_1(t)-u_2(t))\mu_k \psi_{k} \|^2_{H_2^{-1}} dt\right)^\frac{1}{2}\right]\\
			&\leq C \mathbb{E} \left[\left(\int_{0}^{T}\|u_1(t)-u_2(t)\|^2_{H_2^{-1}} \|u_1(t)-u_2(t) \|^2_{H_2^{-1}} dt\right)^\frac{1}{2}\right]\\
			&\leq C \mathbb{E} \left[ \left(\sup\limits_{0\leq t \leq T}\|u_1(t)-u_2(t)\|_{H_2^{-1}}^2\right)^\frac{1}{2} \left(\int_{0}^{T}\|u_1(t)-u_2(t) \|_{H_2^{-1}}^2 dt\right)^\frac{1}{2}\right]\\
			&\leq \frac{1}{4} \mathbb{E} \left[ \sup\limits_{0\leq t \leq T}\|u_1(t)-u_2(t)\|^2_{H_2^{-1}}\right] + C \mathbb{E}\left[  \int_{0}^{T}\|u_1(t)-u_2(t) \|^2_{H_2^{-1}} dt\right].
		\end{aligned}
	\end{equation}
	For $a\geq 2$, combining the estimates for \( J_1 \) and \( J_2 \), we deduce that
	\begin{align*}
		&\quad\frac{1}{4} \mathbb{E} \left[\sup\limits_{0\leq t\leq T} \|u_1(t)-u_2(t)\|_{H_2^{-1}}^2\right] + \frac{1}{4}\mathbb{E}\left[\int_{0}^{T} \|u_1(t)-u_2(t)\|_{L^{m+1}}^{m+1} dt\right]\\
		&\leq C(a,R_1,R_2,T)\mathbb{E} \left[\sup\limits_{0\leq t\leq T} \|\xi_1(t)-\xi_2(t)\|_{H_2^{-1}}^2\right]^{\theta\frac{m+1}{2m}} + C \mathbb{E}\left[\int_{0}^{T} \|u_1(t) - u_2(t)\|^2_{H_2^{-1}}dt\right].
	\end{align*}
	By Gronwall's inequality, we have that there exists a constant $C=C(a,R_1,R_2,T)>0$ such that 
	\begin{align*}
		\mathbb{E} \left[\sup\limits_{0\leq t\leq T} \|u_1(t)-u_2(t)\|_{H_2^{-1}}^2\right] \leq C(a,R_1,R_2,T)\mathbb{E}\left[\sup\limits_{0\leq t\leq T} \|\xi_1(t)-\xi_2(t)\|_{H_2^{-1}}^2\right]^{\theta\frac{m+1}{2m}}.
	\end{align*}
	Similarly, for $1<a<2$, 
	\begin{align*}
		\mathbb{E} \left[\sup\limits_{0\leq t\leq T} \|u_1(t)-u_2(t)\|_{H_2^{-1}}^2\right] &\leq C(a,R_1,R_2,T) \mathbb{E}\left[\sup\limits_{0\leq t\leq T} \|\xi_1(t)-\xi_2(t)\|_{H_2^{-1}}^2\right]^{\theta\frac{m+1}{2m}}\\
		&\quad+ C(a,R_1,R_2,T) \mathbb{E}\left[\sup\limits_{0\leq t\leq T} \|\xi_1(t)-\xi_2(t)\|_{H_2^{-1}}^2\right]^{(a-1)\theta\frac{m+1}{2m}}.
	\end{align*}
	These estimates imply that the operator $\mathcal{T}$ is continuous on the space $\mathbb{X}_\mathfrak{A}(R_1,R_2)$ in both cases where $1<a<2$ and $a\geq 2$. 
\end{proof}

We next verify that $\mathcal{T}$ is a compact operator. To this end, it suffices to show that $\mathcal{T}$ maps bounded sets into precompact sets, which is established by verifying the uniform boundedness and equicontinuity conditions required by the Ascoli-Arzelà theorem. The precise estimates are given in the following proposition.
\begin{proposition} \label{prop:3}
	For any $u_0 \in H_2^{-1}(\mathcal{O})$ satisfying $\mathbb{E}\left[ \|u_0\|_{L^{m+1}}^{m+1} \right] < \infty$, and all $R_1 > 0$ and $R_2 > 0$, it holds that
	
		\noindent (\romannumeral 1). there exists a constant $C > 0$ such that for any $\xi \in \mathbb{X}_\mathfrak{A}(R_1,R_2)$, we have
		\[\sup_{0 \leq t \leq T} \mathbb{E} \|\mathcal{T}(\xi)\|_{L^{m+1}}^{m+1} \leq C .\]
		(\romannumeral 2). there exists a constant $C = C(T, m, R_1, R_2) > 0$ such that for any $0 \leq t_1 < t_2 \leq T$ and $\xi \in \mathbb{X}_\mathfrak{A}(R_1,R_2)$ we have
		\[\mathbb{E} \|\mathcal{T}\xi(t_1) - \mathcal{T}\xi(t_2)\|_{H_2^{-1}}^2 \leq C |t_1 - t_2|.\]

\end{proposition}
\begin{proof}
	\noindent\text{(\romannumeral 1). } Since we have already shown that $\mathcal{T}$ maps $\mathbb{X}_\mathfrak{A}(R_1,R_2)$ i.e. 
	$$ \mathbb{E}\sup\limits_{0\leq t \leq T} \|\mathcal{T}\xi\|_{L^{m+1}}^{m+1}  \leq  R_1, $$ 
	and by Jensen's inequality it follows that 
	$$ \sup\limits_{0\leq t \leq T} \mathbb{E}\|\mathcal{T}\xi\|_{L^{m+1}}^{m+1}  < \mathbb{E}\sup\limits_{0\leq t \leq T} \|\mathcal{T}\xi\|_{L^{m+1}}^{m+1},  $$
	we conclude that 
	$$ \sup\limits_{0\leq t \leq T} \mathbb{E}\|\mathcal{T}\xi\|_{L^{m+1}}^{m+1} \leq R_1. $$
	
	\noindent\text{(\romannumeral 2). } Applying the It\^o formula to the function $\mathcal{T}(u):=\|u(t) -u_0\|_{H_2^{-1}}^2$, we have
	\begin{equation}
		\label{prop31}
		\begin{aligned}
			&\quad\frac{1}{2}\|u(t) - u_0\|_{H_2^{-1}}^2 \\
			&= \int_0^t  {}_{V^*}\langle \Delta u^{[m]}(s), u(s)-u_0 \rangle_V 
			- \chi  _{V^*}\left\langle  \divergence(\xi(s) \nabla v_\xi^{[a]}(s)), u(s)-u_0 \right\rangle_V  ds \\
			&\quad +  \frac{1}{2}\int_0^t \|u(s)\|^2_{L_2(L^2,H_2^{-1})}  ds +  \int_0^t  \langle u(s) , u(s)dW(s)\rangle_{H_2^{-1}}^2  ds\\
			&\leq \int_0^t  {}_{V^*}\langle \Delta( u^{[m]}(s)-u_0^m), u(s)-u_0 \rangle_V ds +\int_{0}^{t} {}_{V^*}\langle \Delta u_0^m, u(s)-u_0 \rangle_V ds\\
			&\quad- \chi\int_0^t {}_{V^*}\left\langle  \divergence(\xi(s) \nabla v_\xi^{[a]}(s)), u(s)-u_0 \right\rangle_V  ds  + C \int_0^t \|u(s)\|^2_{H_2^{-1}}  ds 	\\
			&\quad+  \int_0^t  \langle u(s) , u(s)dW(s)\rangle_{H_2^{-1}}^2  ds\\
			&\leq - \int_{0}^{t} \|u(s) -u_0\|_{L^{m+1}}^{m+1} ds +\int_{0}^{t} {}_{V^*}\langle \Delta u_0^m, u(s)-u_0 \rangle_V ds  \\
			&\quad-   \chi \int_0^t {}_{V^*}\left\langle  \divergence(\xi(s) \nabla v_\xi^{[a]}(s)), u(s)-u_0 \right\rangle_V  ds \\
			&\quad + C\int_{0}^{t}  \|u(s)\|_{H_2^{-1}}^2 ds + 2 \sum_{k=1}^{\infty} \int_{0}^{t} \langle u(s),\mu_k u(s) \psi_{k}\rangle_{H_2^{-1}} d\beta_k(s).
		\end{aligned}
	\end{equation}
	Rearranging terms, we obtain
	\begin{equation}
		\begin{aligned}
			&\quad\frac{1}{2}\|u(t) - u_0\|_{H_2^{-1}}^2 +  \int_{0}^{t} \|u(s) -u_0\|_{L^{m+1}}^{m+1} ds \\
			&\leq  \int_{0}^{t} {}_{V^*}\langle \Delta u_0^m, u(s)-u_0 \rangle_V ds  -   \chi \int_0^t {}_{V^*}\left\langle  \divergence(\xi(s) \nabla v_\xi^{[a]}(s)), u(s)-u_0 \right\rangle_V  ds\\
			&\quad + \sum_{k=1}^{\infty} \int_{0}^{t} \langle u(s),\mu_k u(s) \psi_{k}\rangle_{H_2^{-1}} d\beta_k(s) + C \int_{0}^{t}  \|u(s)\|_{H_2^{-1}}^2 ds\\
			&:= K_1(t) +K_2(t) +K_3(t) + C\int_{0}^{t}  \|u(s)\|_{H_2^{-1}}^2 ds.
		\end{aligned}
	\end{equation}
	Applying the H\"older inequality, Young's inequality and \eqref{dual:V*toV} gives that
	\begin{equation*}
		\begin{aligned}
			\mathbb{E}\left[\sup\limits_{0\leq t\leq r} K_1(t)\right] &\leq  \mathbb{E} \left[\sup\limits_{0\leq t\leq r}\int_{0}^{t} \int_{\mathcal{O}} u_0^m (u(s)-u_0) dxds\right]\\
			&\leq \mathbb{E} \left[\int_{0}^{r} \|u(s) -u_0\|_{L^{m+1}}\|u_0\|^m_{L^{m+1}} ds\right] \\
			&\leq \mathbb{E} \left[\left(\int_{0}^{r} \|u(s) -u_0\|_{L^{m+1}}^{m+1} ds\right)^{\frac{1}{m+1}} \left(\int_{0}^{r} \|u_0\|^{m+1}_{L^{m+1}} ds\right)^{\frac{m}{m+1}}\right] \\
			&\leq \frac{1}{2}\mathbb{E} \left[\int_{0}^{r} \|u(s) -u_0\|_{L^{m+1}}^{m+1} ds\right] + C\mathbb{E}\left[ \int_{0}^{r} \|u_0\|^{m+1}_{L^{m+1}} ds\right] \\
			&\leq \frac{1}{2}\mathbb{E} \left[\int_{0}^{r} \|u(s) -u_0\|_{L^{m+1}}^{m+1} ds\right] +C r \mathbb{E} \left[\|u_0\|^{m+1}_{L^{m+1}} \right].
		\end{aligned}
	\end{equation*}
	By \eqref{dual:V*toV}, we get
	\begin{equation*}
		\begin{aligned}
			\mathbb{E}\left[\sup\limits_{0\leq t\leq r} K_2(t)\right] &=  \chi\mathbb{E} \left[\sup\limits_{0\leq t\leq r}\int_{0}^{t} \int_{\mathcal{O}} (-\nabla)^{-1} (\xi(s) \nabla v_\xi^{[a]}(s)) (u(s)-u_0) dxds\right]\\
			&\leq  \chi\mathbb{E} \left[\int_{0}^{r} \|\xi(s) \nabla v_\xi^{[a]}(s)\|_{H_{\frac{m+1}{m}}^{-1}} \|u(s)-u_0\|_{L^{m+1}} ds\right].
		\end{aligned}
	\end{equation*}
	By the Young's inequality and the Sobolev embedding $L^1(\mathcal{O})\hookrightarrow H_{\frac{m+1}{m}}^{-1}(\mathcal{O})$, we have
	\begin{equation*}
		\begin{aligned}
			\mathbb{E}\left[\sup\limits_{0\leq t\leq r} K_2(t)\right] &\leq  C\mathbb{E} \left[\int_{0}^{r} \|\xi(s) \nabla v_\xi^{[a]}(s)\|^{\frac{m+1}{m}}_{L^1} ds\right] + \frac{1}{4}\mathbb{E}\left[\int_{0}^{r} \|u(s)-u_0\|^{m+1}_{L^{m+1}} ds\right]\\
			&\leq Cr\mathbb{E} \left[\sup\limits_{0\leq s\leq r}\|\xi(s)\|^{m+1}_{L^{m+1}} \right]^{\frac{a+1}{m}} + \frac{1}{4}\mathbb{E}\left[\int_{0}^{r} \|u(s)-u_0\|^{m+1}_{L^{m+1}} ds\right].
		\end{aligned}
	\end{equation*}
	Then we estimate $K_3(t)$ by the Burkholder-Davis-Gundy inequality:
	\begin{equation*}
		\begin{aligned}
			\mathbb{E}\left[\sup\limits_{0\leq t\leq r} K_3(t)\right] &\leq C \mathbb{E} \left[\left(\int_{0}^{r}\|u(s) -u_0\|_{H_2^{-1}}^2 \|u(s)\|^2_{L_2(L^2(\mathcal{O}),H_2^{-1}(\mathcal{O}))} ds\right)^{\frac{1}{2}}\right]\\
			&\leq C \mathbb{E} \left[\left(\sup\limits_{0\leq s\leq r}\|u(s) - u_0\|_{H_2^{-1}}^2  \right)^\frac{1}{2} \left( \int_{0}^{r}\|u(s) \|_{H_2^{-1}}^2 ds\right)^{\frac{1}{2}}\right]\\
			&\leq \frac{1}{4} \mathbb{E} \left[\sup\limits_{0\leq s\leq r}\|u(s) -u_0\|_{H_2^{-1}}^2 \right] +C \mathbb{E} \left[ \int_{0}^{r}\|u(s) \|_{H_2^{-1}}^2 ds \right].
		\end{aligned}
	\end{equation*}
	Combining the estimates for \( K_1 \), \( K_2 \), and \( K_3 \), we rearrange to obtain
	\begin{align*}
		&\quad\frac{1}{4} \mathbb{E} \left[\sup\limits_{0\leq s\leq r}\|u(s) -u_0\|_{H_2^{-1}}^2 \right] +\frac{1}{4} \mathbb{E}\left[\int_{0}^{r} \|u(s)-u_0\|^{m+1}_{L^{m+1}} ds\right] \\
		&\leq r (C\mathbb{E} \left[\|u_0\|^{m+1}_{L^{m+1}} \right] + C(R_1,n,m)) +C r R_2\\
		&:= C(a,m,R_1,R_2) r.
	\end{align*}
	The case  \( t_1 \neq 0 \) follows analogously. Since \(\mathbb{E} \sup\limits_{0\leq s\leq T} \|u(s)\|_{L^{m+1}}^{m+1} < R_1 \), it follows that for any $0\leq t_1 < t_2 \leq T$, we have 
	$$\mathbb{E} \| u(t_1) -u(t_2)\|_{H_2^{-1}}^2 \leq C|t_1-t_2|.$$
\end{proof}

Therefore, we can conclude the existence of at least one fixed point $u \in \mathbb{X}_{\mathfrak{A}}(R_1, R_2)$ such that $\mathcal{T}(u) = u$. By the definition of our operator $\mathcal{T}$, this fixed point is precisely a martingale solution to the stochastic chemotaxis system \eqref{system}. This completes the proof of Theorem \ref{theorem1}.

\section{Preservation of Non-negativity}
Having established the existence of a global martingale solution and the corresponding a priori estimates in Theorem~\ref{theorem1}, we now verify that the solution remains biologically admissible. More precisely, we show that the positive cone is invariant under the stochastic dynamics.

\begin{theorem}\label{theorem2}
	Assume $a\geq1$ and $m\geq2a+1$. Let the initial data
	$u_0\in H^{-1}_2(\mathcal O)$ satisfy
	\[
	\mathbb E\big[\|u_0\|_{L^{m+1}}^{m+1}\big]<\infty,
	\]
	and $u_0\geq0$ almost surely. Then
	\[
	u(t,x)\geq0,
	\quad
	dt\otimes dx\otimes d\mathbb P\text{-a.e. on }
	[0,T]\times\mathcal O\times\Omega .
	\]
\end{theorem}
\begin{proof}
	We follow the argument of \cite[Theorem 2.6.2]{barbu2016stochastic} and adapt it to the porous-medium diffusion case. For any $r\in\mathbb R$, we set
	\[
	r^+:=\max\{r,0\},
	\qquad
	r^-:=\max\{-r,0\}.
	\]
	
	Let us define, for $\delta\in(0,1)$,
	\[
	g_\delta(r)
	:=
	\frac{r^2}{\delta+r},
	\qquad
	r\in(-\delta,\infty),
	\]
	and
	\[
	G_\delta(r)
	:=
	g_\delta((r)^-),
	\qquad r\in\mathbb R .
	\]
	Then $G_\delta\in C^2(\mathbb R)$ and
	\[
	G_\delta(r)=G_\delta'(r)=G_\delta''(r)=0,
	\qquad r\geq0,
	\]
	while
	\[
	|G_\delta'(r)|\leq 2r^-,
	\qquad
	|G_\delta''(r)|\leq 8,
	\qquad r\in\mathbb R .
	\]
	
	Now define
	\[
	\phi_\delta:L^2(\mathcal O)\rightarrow\mathbb R
	\]
	by
	\[
	\phi_\delta(u)
	:=
	\int_{\mathcal O}G_\delta(u(x))\,dx .
	\]
	Then $\phi_\delta$ is twice Gâteaux differentiable on $L^2(\mathcal O)$.
	
	We now apply Itô's formula to $\phi_\delta(u(t))$. More precisely, one first applies Itô's formula to a suitable approximation of the equation and then passes to the limit. For convenience, we omit this standard approximation argument and write the resulting identity directly.
	
	Since the stochastic integral is a martingale, we obtain
	\begin{align}
		\mathbb E[\phi_\delta(u(t))]
		&=
		\mathbb E[\phi_\delta(u_0)]
		+
		\mathbb E
		\int_0^t
		\int_{\mathcal O}
		G_\delta'(u)
		\Delta(u^{[m]})\,dx\,ds
		\notag\\
		&\quad
		-\chi
		\mathbb E
		\int_0^t
		\int_{\mathcal O}
		G_\delta'(u)
		\nabla\cdot
		(u\nabla v^{[a]})
		\,dx\,ds
		\notag\\
		&\quad
		+\frac12
		\sum_{k=1}^{\infty}
		\mu_k^2
		\mathbb E
		\int_0^t
		\int_{\mathcal O}
		G_\delta''(u)
		u^2\psi_k^2\,dx\,ds .
		\label{eq:nonnegative-ito}
	\end{align}
	
	Since $u_0\geq0$ almost surely and $G_\delta(r)=0$ for $r\geq0$,
	we have
	\[
	\phi_\delta(u_0)=0 .
	\]
	
	For the porous-medium diffusion term, using the Neumann boundary condition and integration by parts, we obtain
	\begin{align*}
		\int_{\mathcal O}
		G_\delta'(u)\Delta(u^{[m]})\,dx
		&=
		-\int_{\mathcal O}
		G_\delta''(u)
		\nabla u\cdot
		\nabla(u^{[m]})\,dx .
	\end{align*}
	Since
	\[
	\nabla(u^{[m]})
	=
	m|u|^{m-1}\nabla u ,
	\]
	we get
	\[
	-\int_{\mathcal O} G_\delta''(u)\nabla u\cdot\nabla(u^{[m]})\,dx = -m\int_{\mathcal O} G_\delta''(u) |u|^{m-1}|\nabla u|^2\,dx \leq0 .
	\]
	
	Next, we estimate the chemotactic term. Using integration by parts and the Neumann boundary condition, we have
	\begin{align}
		-\chi
		\int_{\mathcal O}
		G_\delta'(u)
		\nabla\cdot
		(u\nabla v^{[a]})
		\,dx
		&
		=
		\chi
		\int_{\mathcal O}
		\nabla G_\delta'(u)
		\cdot
		u\nabla v^{[a]}
		\,dx
		\notag\\
		&
		=
		\chi
		\int_{\mathcal O}
		G_\delta''(u)
		u
		\nabla u\cdot
		\nabla v^{[a]}
		\,dx .
		\label{eq:chemotaxis-positive}
	\end{align}
	
	By Young's inequality,
	\begin{align}
		&\quad\chi
		\int_{\mathcal O}
		G_\delta''(u)
		u
		\nabla u\cdot
		\nabla v^{[a]}
		\,dx
		\notag\\
		&\leq
		\frac{m}{2}
		\int_{\mathcal O}
		G_\delta''(u)
		|u|^{m-1}
		|\nabla u|^2dx
		+C
		\int_{\mathcal O}
		G_\delta''(u)
		|u|^{3-m}
		|\nabla v^{[a]}|^2dx .
	\end{align}
	
	Since $G_\delta''(u)$ is supported on
	$-\delta<u<0$, we have
	\[
	|u|^{3-m}G_\delta''(u)
	\leq C_\delta .
	\]
	Therefore,
	\[
	\chi
	\int_{\mathcal O}
	G_\delta''(u)
	u
	\nabla u\cdot
	\nabla v^{[a]}
	dx
	\leq
	\frac{m}{2}
	\int_{\mathcal O}
	G_\delta''(u)
	|u|^{m-1}
	|\nabla u|^2dx
	+
	C_\delta
	\|\nabla v^{[a]}\|_{L^\infty}^2 .
	\]
	
	By the elliptic estimate \eqref{ineq:v_infty} and the uniform estimate
	\eqref{est:thm_1},
	\[
	\|\nabla v^{[a]}\|_{L^\infty} \leq C\|u\|_{L^{m+1}}^a ,
	\]
	and hence
	\[
	\mathbb E \int_0^t \|\nabla v^{[a]}(s)\|_{L^\infty}^2ds <\infty .
	\]
	
	For the Itô correction term, by the property of $G_\delta''$, we have
	\[
	u^2G_\delta''(u) \leq C\delta^2.
	\]
	Using the noise assumption \eqref{assum:noise}, we obtain
	\begin{align}
		&\frac12 \sum_{k=1}^{\infty} \mu_k^2 \mathbb E \int_0^t \int_{\mathcal O} G_\delta''(u) u^2\psi_k^2dxds	\leq C\delta^2.
	\end{align}
	
	Combining the above estimates with
	\eqref{eq:nonnegative-ito}, and absorbing the diffusion term, we obtain
	\[
	\mathbb E[\phi_\delta(u(t))]\leq C\delta.
	\]
	
	Finally, by the definition of $G_\delta$,
	\[
	G_\delta(r)\rightarrow (r^-)^2,\qquad\text{as }\delta\rightarrow0.
	\]
	Using Fatou's lemma, we obtain
	\[
	\mathbb E\int_{\mathcal O}(u^-(t,x))^2dx\leq0.
	\]
	Therefore,
	\[
	u^-(t,x)=0,	\quad	dt\otimes dx\otimes d\mathbb P\text{-a.e.},
	\]
	which implies
	\[
	u(t,x)\geq0,\quad dt\otimes dx\otimes d\mathbb P\text{-a.e.}
	\]
	This completes the proof.
\end{proof}

\section{Conclusion}

In this paper, we have studied a stochastic Keller--Segel system with porous medium diffusion and nonlinear chemotactic sensitivity. The model is motivated by cell aggregation phenomena in randomly fluctuating biological environments, especially in situations where cells move through crowded or porous media and respond nonlinearly to chemical signals. Compared with the classical Keller--Segel system, the present model incorporates the nonlinear diffusion term $\Delta u^{[m]}$, which describes density-dependent dispersal of the cell population, and the nonlinear chemotactic drift $\nabla\cdot(u\nabla v^{[a]})$, which reflects a power-law response of cells to the perceived chemoattractant. The inclusion of multiplicative noise further captures random environmental effects and stochastic fluctuations in population dynamics. Therefore, the model provides a more realistic mathematical description of chemotactic movement in complex biological media.

The main result of this paper is the global existence of martingale solutions to the stochastic Keller--Segel system under the condition $a\geq1$ and $m\geq2a+1$. We also establish uniform a priori estimates in the spaces $H^{-1}_2(\mathcal O)$ and $L^{m+1}(\mathcal O)$, together with gradient estimates associated with the porous-medium diffusion. Moreover, when the initial cell density is nonnegative, the solution remains nonnegative for all time. This result shows the balance between nonlinear chemotactic aggregation and density-dependent diffusion. In particular, when the chemotactic response becomes stronger, namely when the exponent $a$ increases, a stronger porous-medium diffusion is required to prevent excessive aggregation and ensure the global-in-time existence of solutions.

Due to the degeneracy of the porous-medium diffusion, the strong nonlinearity of the chemotactic drift, and the low regularity caused by the multiplicative stochastic perturbation, the analysis requires establishing suitable estimates for the nonlinear coupling. The main difficulty lies in estimating the nonlinear chemotactic term $\xi\nabla v_\xi^{[a]}$ arising from the power-law sensitivity, whose interaction with the degenerate diffusion and stochastic perturbation prevents a direct application of standard variational SPDE framework. To overcome this difficulty, we construct a decoupled auxiliary system, derive uniform energy estimates, and apply a stochastic Schauder--Tychonoff fixed point theorem to obtain a martingale solution. Compared with the stochastic porous-medium chemotaxis model studied in \cite{mukherjee2025martingale} and the linear chemotactic sensitivity case considered in \cite{wang2025global}, the main novelty of this work is the power-law chemotactic response $v^{[a]}$, which leads to the new balance condition $m\geq 2a+1$.

In conclusion, our results show that stochastic perturbations, nonlinear diffusion, and nonlinear biological sensing can be combined in a mathematically consistent framework for chemotaxis. The obtained global martingale solution and non-negativity result provide a rigorous foundation for the proposed biological model.
\section{Acknowledgments}
This work is supported by the National Natural Science Foundation of China (No. 12271269 and No. 12471141) and Fundamental Research Funds for the Central Universities.

\bibliographystyle{spbasic}
\bibliography{refer}

\end{document}